\documentclass{amsart}
\usepackage{epsfig}
\usepackage{url}	

\newtheorem{theorem}{Theorem}[section]
\newtheorem{lemma}[theorem]{Lemma}
\newtheorem{corollary}[theorem]{Corollary}

\theoremstyle{definition}

\newtheorem{example}[theorem]{Example}

\theoremstyle{remark}

\numberwithin{equation}{section}

\begin{document}

\title[Arbitrary powers of the arcsine]{Integrals involving 
arbitrary powers of the arcsine, with applications to infinite series}

\author{Karl Dilcher}
\address{Department of Mathematics and Statistics\\
 Dalhousie University\\
         Halifax, Nova Scotia, B3H 4R2, Canada}
\email{dilcher@mathstat.dal.ca}

\author{Christophe Vignat}
\address{CentraleSup\'elec, Universit\'e Paris-Saclay, Gif-sur-Yvette, France and Department of
Mathematics, Tulane University, New Orleans, LA 70118, USA}
\email{cvignat@tulane.edu}
\thanks{The first author was supported in part by the Natural Sciences and 
Engineering Research Council of Canada.}

%\date{October 25, 2024}

\setcounter{equation}{0}

\begin{abstract}
Using appropriate power series evaluations, we determine all moments of 
arbitrary positive powers of the arcsine. As consequences we evaluate  several
doubly infinite classes of power series involving central binomial coefficients
and generalized multiple harmonic sums. By specializing the variable involved,
we then evaluate classes of numerical sequences, mostly in terms of powers 
of $\pi$. Finally, we obtain limit expressions for arbitrary powers of $\pi$. 
\end{abstract}

\maketitle

\section{Introduction}\label{sec:1}

The arcsine function is of particular interest in the study of series for $\pi$
and expressions involving $\pi$. For instance, the well-known power series
\begin{equation}\label{1.1}
\arcsin{x} = \sum_{k=0}^\infty\frac{\binom{2k}{k}}{4^k}\frac{x^{2k+1}}{2k+1},
\qquad
(\arcsin{x})^2=\sum_{k=1}^\infty\frac{4^k}{\binom{2k}{k}k}\frac{x^{2k}}{2k}
\end{equation}
lead to series representations for $\pi$ and $\pi^2$ by setting, for instance, 
$x=\frac{1}{2}, \frac{1}{2}\sqrt{2}, \frac{1}{2}\sqrt{3}$, 1, or other related 
algebraic numbers. These and many other results can be found, for instance, 
in the interesting paper \cite{Le} by D.~H.~Lehmer.

Partly inspired by Lehmer's use of integration as a method for obtaining series
identities,  we made a systematic study in \cite{DV} of what amounts to 
arbitrary numbers of repeated integrations of $\arcsin{x}$ and a few of its 
powers. For instance, we showed that for all even $n\geq 0$ we have
\begin{equation}\label{1.2}
\sum_{k=0}^\infty\frac{\binom{2k}{k}}{4^k}\frac{x^{2k+1}}{2k+n+1}
=\frac{\binom{n}{n/2}}{(2x)^n}\left(\arcsin{x}-\frac{\sqrt{1-x^2}}{2}
\sum_{j=0}^{\frac{n-2}{2}}\frac{(2x)^{2j+1}}{(2j+1)\binom{2j}{j}}\right),
\end{equation}
with similar results involving squares, cubes, and fourth powers of 
$\arcsin{x}$, also for odd $n\geq 1$. It is clear that the first identity
of \eqref{1.1} is the case $n=0$ of \eqref{1.2}. Furthermore, with $x=1$ we get
\begin{equation}\label{1.3}
\sum_{k=0}^\infty\frac{\binom{2k}{k}}{4^k(2k+n+1)}
=\binom{n}{n/2}\frac{\pi}{2^{n+1}}\quad(n\;\hbox{even});
\end{equation}
see \cite{DV} for further details. Although this is not relevant here, we
showed in \cite{DV2} that \eqref{1.3} can be extended to all complex 
parameters $n$.

One of the main ingredients for all these results was the following lemma 
(Corollary~2.2 in \cite{DV}) which will also play an important role in the 
current paper. 

\begin{lemma}\label{lem:1.1}
Suppose that $f(x)$ is either an even or an odd function, analytic in a 
neighbourhood of $0$, with power series expansion
\begin{equation}\label{1.4}
f(x)=\sum_{k=1-\delta}^{\infty}c_k\cdot\frac{x^{2k+\delta}}{2k+\delta}
\qquad(|x|<R),
\end{equation}
where $\delta$ is either $0$ or $1$, and $R$ is the radius of convergence.
Then, for any $n\geq 0$ and real $x$ with $|x|<R$, we have
\begin{equation}\label{1.5}
\sum_{k=1-\delta}^{\infty}c_k\cdot\frac{x^{2k+\delta+n+1}}{2k+\delta+n+1}
=f(x)x^{n+1}-(n+1)\int_0^x t^n f(t)dt.
\end{equation}
\end{lemma}

The key to the results in \cite{DV} was the evaluation of the integral on the
right of \eqref{1.5}, with $f(t)=(\arcsin{t})^q$ for $q=1,\ldots, 4$.
The main purpose of the current paper is to extend this to all positive 
integers $q\geq 1$ by using a different method. 

For integers $n\geq 0$ and $q\geq 1$, we denote
\begin{equation}\label{1.6}
I_q^{(n)}(x) := \int_0^x t^{n}(\arcsin{t})^q\,dt.
\end{equation}
It appears that apart from the results in \cite{DV} mentioned above, only the 
case $q=1$ had previously been evaluated for all 
$n\geq 0$, separately for even and odd $n$, as identities 1.7.3.2 and 
1.7.3.3 in \cite{PrE}. In the current paper we will also need to treat the cases
of even and odd $n$ separately, as well as distinguish between even and odd
$q$. The corresponding four identities are stated in Section~\ref{sec:2}.

The main reason for wishing to obtain evaluations for the integrals in 
\eqref{1.6} for arbitrary integers $q$ is the existence of the following two
classes of series expansions for $(\arcsin{x})^p$. They can be found in 
\cite{BC} or \cite[p.~124]{Sch}. Using a slight variant of the notation and 
normalization from \cite{BC}, we have for integers $p\geq 0$,
\begin{equation}\label{1.7}
(\arcsin{x})^{2p+1} = (2p+1)!
\sum_{k=0}^{\infty}\frac{\binom{2k}{k}G_p(k)}{4^k}\cdot\frac{x^{2k+1}}{2k+1},
\end{equation}
where $G_0(k)=1$ and for $p\geq 1$,
\begin{equation}\label{1.8}
G_p(k)=\sum_{n_1=0}^{k-1}\frac{1}{(2n_1+1)^2}
\sum_{n_2=0}^{n_1-1}\frac{1}{(2n_2+1)^2}\cdots
\sum_{n_p=0}^{n_{p-1}-1}\frac{1}{(2n_p+1)^2}.
\end{equation}
Similarly, for integers $p\geq 1$,
\begin{equation}\label{1.9}
(\arcsin{x})^{2p} = (2p)!
\sum_{k=1}^{\infty}\frac{4^kH_p(k)}{\binom{2k}{k}2k}\cdot\frac{x^{2k}}{2k},
\end{equation}
where $H_1(k)=1$ and
\begin{equation}\label{1.10}
H_{p+1}(k)=\sum_{n_1=1}^{k-1}\frac{1}{(2n_1)^2}
\sum_{n_2=1}^{n_1-1}\frac{1}{(2n_2)^2}\cdots
\sum_{n_p=1}^{n_{p-1}-1}\frac{1}{(2n_p)^2}.  
\end{equation}
For some further remarks on the sums $G_p(k)$ and $H_{p+1}(k)$, see 
Section~\ref{sec:5}.3 below.
It is clear that with $p=0$ in \eqref{1.7} and $p=1$ in \eqref{1.9} we get the
two identities in \eqref{1.1}. For the next smallest cases, namely $G_1(k)$ 
and $H_2(k)$, see \eqref{4.5} below.

In Section~\ref{sec:2} we will state and prove the main result of this paper,
namely the evaluation of the integral in \eqref{1.6} for all integer parameters
$n\geq 0$ and $q\geq 1$. We then use this together with Lemma~\ref{lem:1.1} and
the expansions \eqref{1.7}, \eqref{1.9} in Section~\ref{sec:3} to obtain 
various special cases, along with some specific examples. In Section~\ref{sec:4}
we use our main result to obtain certain limit expressions for all positive
integer powers of $\pi$. We finish with some remarks and further consequences 
in Section~\ref{sec:5}.

\section{The main results}\label{sec:2}

To state our main results, we need to recall some mathematical objects and 
introduce some notation. The Chebyshev polynomials of the first kind, $T_n(x)$,
and of the second kind, $U_n(x)$, are usually defined by 
\begin{equation}\label{2.1}
T_n(\cos{\theta})=\cos(n\theta)\quad\hbox{and}\quad
U_n(\cos{\theta})=\frac{\sin((n+1))}{\sin{\theta}}
\end{equation}
(see, e.g., \cite[p.~993ff.]{GR} or \cite{Ri}). The first few Chebyshev 
polynomials are $T_0(x)=1$, $T_1(x)=x$, $T_2(x)=2x^2-1$, $T_3(x)=4x^3-3x$, and 
$U_0(x)=1$, $U_1(x)=2x$, $U_2(x)=4x^2-1$, $U_3(x)=8x^3-4x$.

For greater ease of notation in what follows, we set
\[
a=a(x):=\arcsin{x},
\]
and we denote the partial sums of the Maclaurin series for cosine and sine by
\begin{equation}\label{2.2}
c_n(z):=\sum_{j=0}^n(-1)^j\frac{z^{2j}}{(2j)!},\qquad
s_n(z):=\sum_{j=0}^n(-1)^j\frac{z^{2j+1}}{(2j+1)!}.
\end{equation}
We are now ready to state our main result, which evaluates the integral in
\eqref{1.6}.

\begin{theorem}\label{thm:2.1}
For all integers $\ell\geq 0$ and $p\geq 0$ we have, with $a=\arcsin{x}$,
\begin{align}
I_{2p+1}^{(2\ell)}(x)&=\frac{x^{2\ell+1}}{2\ell+1}a^{2p+1}
+\frac{(-1)^p(2p+1)!}{4^\ell(2\ell+1)}\sum_{k=0}^\ell\binom{2\ell+1}{\ell-k}
\frac{1}{(2k+1)^{2p+1}}\label{2.3} \\
&\times\left(s_{p-1}((2k+1)a)T_{2k+1}(x)
+c_p((2k+1)a)U_{2k}(x)\sqrt{1-x^2}-(-1)^k\right),\nonumber \\
I_{2p+1}^{(2\ell-1)}(x)
&=\left(x^{2\ell}-\frac{\binom{2\ell}{\ell}}{4^\ell}\right)
\frac{a^{2p+1}}{2\ell}
+\frac{(-1)^p(2p+1)!}{2^{2\ell-1}2\ell}\sum_{k=1}^\ell\binom{2\ell}{\ell-k}
\frac{1}{(2k)^{2p+1}}\label{2.4} \\
&\times\left(s_{p-1}(2ka)T_{2k}(x)+c_p(2ka)U_{2k-1}(x)\sqrt{1-x^2}\right),\nonumber 
\end{align}
and for $p\geq 1$,
\begin{align}
I_{2p}^{(2\ell)}(x)&=\frac{x^{2\ell+1}}{2\ell+1}a^{2p}
+\frac{(-1)^p(2p)!}{4^\ell(2\ell+1)}\sum_{k=0}^\ell\binom{2\ell+1}{\ell-k}
\frac{1}{(2k+1)^{2p}}\label{2.5} \\
&\times\left(c_{p-1}((2k+1)a)T_{2k+1}(x)
-s_{p-1}((2k+1)a)U_{2k}(x)\sqrt{1-x^2}\right),\nonumber \\
I_{2p}^{(2\ell-1)}(x)
&=\left(x^{2\ell}-\frac{\binom{2\ell}{\ell}}{4^\ell}\right)\frac{a^{2p}}{2\ell}
+\frac{(-1)^p(2p)!}{2^{2\ell-1}2\ell}\sum_{k=1}^\ell\binom{2\ell}{\ell-k}
\frac{1}{(2k)^{2p}}\label{2.6} \\
&\times\left(c_{p-1}(2ka)T_{2k}(x)
-s_{p-1}(2ka)U_{2k-1}(x)\sqrt{1-x^2}-(-1)^k\right).\nonumber
\end{align}
\end{theorem}

The main idea of the proof of this theorem lies in considering the 
exponential generating function of the
sequence of integrals $(I_q^{(n)}(x))_{q\geq 0}$ from \eqref{1.6}, where
$n\geq 0$ is fixed. With this in mind, we begin with the following lemma.

\begin{lemma}\label{lem:3.1}
For any integer $n\geq 0$, we define the exponential generating function
\begin{equation}\label{3.1}
I^{(n)}(x,w):=\sum_{q=0}^\infty I_q^{(n)}(x)\frac{w^q}{q!}.
\end{equation}
Then, with $a(x)=\arcsin{x}$, we have
\begin{equation}\label{3.2}
I^{(n)}(x,w)=\frac{x^{n+1}}{n+1}e^{wa(x)}
-\frac{w}{n+1}\int_0^{a(x)}(\sin{\theta})^{n+1}e^{w\theta}d\theta.
\end{equation}
\end{lemma}

\begin{proof}
Since $0\leq\arcsin{t}\leq\pi/2$ for $0\leq t\leq 1$, we have by \eqref{1.6},
\[
0\leq I_q^{(n)}(x)\leq \left(\tfrac{\pi}{2}\right)^q\int_0^1 t^ndt
=\frac{1}{n+1}\left(\tfrac{\pi}{2}\right)^q,
\]
and therefore the power series in \eqref{3.1} has radius of convergence of at
least $2/\pi$. We now substitute \eqref{1.6} into \eqref{3.1}, obtaining
\begin{align}
I^{(n)}(x,w)&=\sum_{q=0}^\infty
\left(\int_0^x t^{n}a(t)^q\,dt\right)\frac{w^q}{q!} \label{3.3}\\
&=\int_0^x t^{n}\left(\sum_{q=0}^\infty a(t)^q\frac{w^q}{q!}\right)dt
=\int_0^x t^{n}e^{wa(t)}dt,\nonumber
\end{align}
where the interchange of integration and infinite series is justified since we
are dealing with convergent power series. 

With the change of variables $t=\sin\theta$, we get from \eqref{3.3},
\[
I^{(n)}(x,w)
=\int_0^{a(x)}\left(\sin^n\theta\cdot\cos\theta\right)e^{w\theta}d\theta.
\]
Noting that 
$\frac{d}{d\theta}\sin^{n+1}\theta=(n+1)\sin^n\theta\cdot\cos\theta$,
we find upon integration by parts that
\[
I^{(n)}(x,w)=\left.\frac{\sin^{n+1}\theta}{n+1}e^{w\theta}\right|_0^{a(x)}
-\frac{w}{n+1}\int_0^{a(x)}\sin^{n+1}\theta\cdot e^{w\theta}d\theta,
\]
and \eqref{3.2} follows.
\end{proof}

The integral in \eqref{3.2} plays an important role, and we denote it by
\begin{equation}\label{3.4}
J^{(n)}(x,w):=\int_0^{a(x)}(\sin{\theta})^{n+1}e^{w\theta}d\theta.
\end{equation}
We also use the standard notation
\begin{equation}\label{3.5}
\left[w^m\right]f(w)=c_m\quad\hbox{where}\quad f(w)=\sum_{j=0}^\infty c_jw^j.
\end{equation}
We can now state and prove the next intermediate result.

\begin{lemma}\label{lem:3.2}
For all integers $n\geq 0$ and $q\geq 1$ we have
\begin{equation}\label{3.6}
I_q^{(n)}(x)=\frac{x^{n+1}}{n+1}a(x)^q
-\frac{q!}{n+1}\left[w^{q-1}\right]J^{(n)}(x,w).
\end{equation}
\end{lemma}

\begin{proof}
We note that 
\[
e^{wa(x)} = \sum_{q=0}^\infty a(x)^q\frac{w^q}{q!},
\]
and that by \eqref{3.5} we have
\[
\left[w^q\right]\left(wJ^{(n)}(x,w)\right)=\left[w^{q-1}\right]J^{(n)}(x,w).
\]
The result now follows from \eqref{3.2} and \eqref{3.1} upon multiplying both
sides by $q!$.
\end{proof}

The Chebyshev polynomials appear in the proof of Theorem~\ref{thm:2.1} in a 
natural way through the following identities, which are similar to the 
defining identities in~\eqref{2.1}.

\begin{lemma}\label{lem:3.3}
For any real $x$ and integers $k\geq 0$, we have
\begin{align}
\cos(2k\arcsin{x}) &= (-1)^kT_{2k}(x),\label{3.7}\\
\sin((2k+1)\arcsin{x}) &= (-1)^kT_{2k+1}(x),\label{3.8}\\
\sin(2k\arcsin{x}) &= (-1)^{k+1}U_{2k-1}(x)\sqrt{1-x^2},\label{3.9}\\
\cos((2k+1)\arcsin{x}) &= (-1)^kU_{2k}(x)\sqrt{1-x^2}.\label{3.10}
\end{align}
\end{lemma}

\begin{proof}
We set $x=\sin{a}$ throughout. To prove \eqref{3.7}, we use some basic 
trigonometric identities and the first part of \eqref{2.1} to obtain
\[
T_{2k}(x)=T_{2k}(\sin{a})=T_{2k}(\cos(a-\tfrac{\pi}{2}))
=\cos(2k(a-\tfrac{\pi}{2}))=(-1)^k\cos(2ka),
\]
which is equivalent to \eqref{3.7}. Similarly, we have
\begin{align*}
\sin((2k+1)a)&=\cos((2k+1)a-\tfrac{\pi}{2})
=(-1)^k\cos((2k+1)(a-\tfrac{\pi}{2})) \\
&=(-1)^kT_{2k+1}\left(\cos(a-\tfrac{\pi}{2})\right)
=(-1)^kT_{2k+1}(\sin{a}),
\end{align*}
which is equivalent to \eqref{3.8}. Next, we use the facts that 
\[
\frac{d}{dx}T_n(x)=nU_{n-1}(x),\quad\hbox{and}\quad
\frac{d}{dx}\arcsin{x}=\frac{1}{\sqrt{1-x^2}}
\]
to obtain \eqref{3.9} and \eqref{3.10} by differentiating both sides of 
\eqref{3.7} and \eqref{3.8}, respectively.
\end{proof}

We can now evaluate the integral $J^{(n)}(x,w)$, which is defined in 
\eqref{3.4}. This needs to be done separately for even and odd $n$.

\begin{lemma}\label{lem:3.4}
For all integers $\ell\geq 1$ and with $a=\arcsin{x}$, we have
\begin{align}
J^{(2l-1)}(x,w)&=\frac{\binom{2\ell}{\ell}(e^{wa}-1)}{4^\ell w}
+\frac{e^{wa}}{2^{2\ell-1}}
\sum_{k=1}^\ell\binom{2\ell}{\ell-k}\frac{V_{2k}(x,w)}{w^2+(2k)^2}\label{3.11}\\
&\quad-\frac{1}{2^{2\ell-1}}
\sum_{k=1}^\ell(-1)^k\binom{2\ell}{\ell-k}\frac{w}{w^2+(2k)^2},\nonumber\\
J^{(2l)}(x,w)&= \frac{e^{wa}}{4^{\ell}}
\sum_{k=0}^\ell\binom{2\ell+1}{\ell-k}\frac{V_{2k+1}(x,w)}{w^2+(2k+1)^2}\label{3.12}\\
&\quad+\frac{1}{4^{\ell}}
\sum_{k=0}^\ell(-1)^k\binom{2\ell+1}{\ell-k}\frac{2k+1}{w^2+(2k+1)^2},\nonumber
\end{align}
where
\begin{equation}\label{3.13}
V_n(x,w):=wT_n(x)-nU_{n-1}(x)\sqrt{1-x^2}.
\end{equation}
\end{lemma}

\begin{proof}
The indefinite integral corresponding to \eqref{3.4} can be found in 
\cite{PrE}, separate for even and odd $n$. In particular, the identity
(1.5.49.8) with $a=w$ and $b=1$ reads
\begin{align}
\int\sin^{2\ell}\theta\cdot e^{w\theta}d\theta
&=\frac{\binom{2\ell}{\ell}e^{w\theta}}{4^\ell w} \label{3.14}\\
&+\frac{e^{w\theta}}{2^{2\ell-1}}\sum_{k=1}^\ell(-1)^k\binom{2\ell}{\ell+k}
\frac{w\cos(2k\theta)+2k\sin(2k\theta)}{w^2+(2k)^2}.\nonumber
\end{align}
The definite integral is then clearly \eqref{3.11} if we note that with
\eqref{3.7} and \eqref{3.9}, the right-most numerator in \eqref{3.14} becomes
$V_{2k}(x,w)$. 

Similarly, rewriting the identity (1.5.49.10) in \cite{PrE}, we have
\begin{align}
&\int\sin^{2\ell+1}\theta\cdot e^{w\theta}d\theta \label{3.15}\\
&\qquad=\frac{e^{w\theta}}{4^\ell}\sum_{k=0}^\ell(-1)^k
\binom{2\ell+1}{\ell+k}
\frac{w\sin((2k+1)\theta)-(2k+1)\cos((2k+1)\theta)}{w^2+(2k+1)^2}.\nonumber
\end{align}
The definite integral will now be \eqref{3.12} if we note that with \eqref{3.8}
and \eqref{3.10} the numerator on the right of \eqref{3.15} becomes 
$V_{2k+1}(x,w)$.
\end{proof}

To facilitate evaluating the coefficients of the powers of $w$ in \eqref{3.11}
and \eqref{3.12}, we state and prove another technical lemma.

\begin{lemma}\label{lem:3.5}
For any positive integer $n$ and real $a$, we have
\begin{align}
&\frac{e^{wa}V_n(x,w)}{w^2+n^2} 
=\sum_{p=0}^\infty(-1)^p\left(c_p(na)T_n(x)
-s_p(na)U_{n-1}(x)\sqrt{1-x^2}\right)\frac{w^{2p+1}}{n^{2p+2}}\label{3.16}\\
&\qquad+\sum_{p=0}^\infty(-1)^p\left(s_{p-1}(na)T_n(x)
+s_p(na)U_{n-1}(x)\sqrt{1-x^2}\right)\frac{w^{2p}}{n^{2p+1}},\nonumber
\end{align}
with $c_p(z)$ and $s_p(z)$ as defined in \eqref{2.2} and $s_{-1}(z)=0$.
\end{lemma}

\begin{proof}
We first take the Cauchy product of the two power series
\[
e^{wa}=\sum_{j=0}^\infty\frac{a^j}{j!}w^j\quad\hbox{and}\quad
\frac{1}{w^2+n^2}=\sum_{j=0}^\infty\frac{(-1)^j}{n^{2j+2}}w^{2j},
\]
obtaining
\begin{align*}
\frac{e^{wa}}{w^2+n^2}&=\sum_{p=0}^\infty\bigg(
\sum_{j=0}^p\frac{a^{2j}}{(2j)!}\frac{(-1)^{p-j}}{n^{2p+2-2j}}\bigg)w^{2p}
+\sum_{p=0}^\infty\bigg(\sum_{j=0}^p\frac{a^{2j+1}}{(2j+1)!}
\frac{(-1)^{p-j}}{n^{2p+2-2j}}\bigg)w^{2p+1}\\
&=\sum_{p=0}^\infty\bigg(
\sum_{j=0}^p(-1)^j\frac{(na)^{2j}}{(2j)!}\bigg)\frac{(-1)^pw^{2p}}{n^{2p+2}}\\
&\qquad+\sum_{p=0}^\infty\bigg(\sum_{j=0}^p
(-1)^j\frac{(na)^{2j+1}}{(2j+1)!}\bigg)\frac{(-1)^pw^{2p+1}}{n^{2p+3}}.
\end{align*}
With \eqref{2.2} we now get 
\begin{equation}\label{3.17}
\frac{e^{wa}}{w^2+n^2}=\sum_{p=0}^\infty(-1)^p c_p(na)\frac{w^{2p}}{n^{2p+2}}
+\sum_{p=0}^\infty(-1)^p s_p(na)\frac{w^{2p+1}}{n^{2p+3}}.
\end{equation}
Next, the product of \eqref{3.13} and \eqref{3.17} is
\begin{align*}
&\sum_{p=0}^\infty(-1)^p c_p(na)T_n(x)\frac{w^{2p}}{n^{2p+2}}
+\sum_{p=0}^\infty(-1)^p s_p(na)T_n(x)\frac{w^{2p+2}}{n^{2p+3}}\\
&\qquad-\sum_{p=0}^\infty(-1)^p c_p(na)U_{n-1}(x)\sqrt{1-x^2}\frac{w^{2p}}{n^{2p+1}}\\
&\qquad-\sum_{p=0}^\infty(-1)^p s_p(na)U_{n-1}(x)\sqrt{1-x^2}\frac{w^{2p+1}}{n^{2p+2}},
\end{align*}
and upon shifting the summation in the second sum, we immediately get 
\eqref{3.16}.
\end{proof}

We are now ready to complete the proof.

\begin{proof}[Proof of Theorem~\ref{thm:2.1}]
We distinguish between four cases related to $I_q^{(n)}(x)$ defined in 
\eqref{1.6}.

{\bf Case 1:} $n=2\ell-1$, $\ell\geq 1$. According to \eqref{3.6}, we need to
determine the coefficients of $w^{q-1}$ in \eqref{3.11}. For this purpose, we
require in addition to Lemma~\ref{lem:3.5} also the easy expansions
\begin{align}
\frac{e^{wa}-1}{w} &= \sum_{j=0}^\infty\frac{a^{j+1}}{(j+1)!}w^j,\label{3.18}\\
\frac{w}{w^2+(2k)^2} &= \sum_{j=0}^\infty\frac{(-1)^j}{(2k)^{2j+2}}w^{2j+1}.\label{3.19}
\end{align}

Subcase 1(a): $q=2p+1$ for $p\geq 0$. That is, for a given $p$ we need to
determine the coefficient of $w^{2p}$ on the right of \eqref{3.11}. There is no
contribution from \eqref{3.19}, while \eqref{3.18} and \eqref{3.16} give the
coefficient
\[
\frac{\binom{2\ell}{\ell}a^{2p+1}}{4^\ell(2p+1)!}
-\frac{(-1)^p}{2^{2\ell-1}}\sum_{k=1}^\ell\binom{2\ell}{\ell-k}
\frac{s_{p-1}(2ka)T_{2k}(x)+c_p(2ka)U_{2k-1}(x)\sqrt{1-x^2}}{(2k)^{2p+1}}.
\]
With \eqref{3.6} we then get \eqref{2.4}.

Subcase 1(b): $q=2p$ for $p\geq 1$. In this case, we determine the
coefficient of $w^{2p-1}$ on the right of \eqref{3.11}. From \eqref{3.18},
\eqref{3.16} and \eqref{3.19}, we get the coefficient
\[
\frac{\binom{2\ell}{\ell}a^{2p}}{4^\ell(2p)!}
-\frac{(-1)^p}{2^{2\ell-1}}\sum_{k=1}^\ell\binom{2\ell}{\ell-k}
\frac{c_{p-1}(2ka)T_{2k}(x)-s_p(2ka)U_{2k-1}(x)\sqrt{1-x^2}-(-1)^k}
{(2k)^{2p}}.
\]
With \eqref{3.6}, this now gives \eqref{2.6}.

{\bf Case 2:} $n=2\ell$, $\ell\geq 1$. In this case we need to determine the
coefficients of $w^{q-1}$ in \eqref{3.12}. This will be similar to Case~1, but
\eqref{3.18} will no longer be required, and in place of \eqref{3.19} we use
the geometric series
\begin{equation}\label{3.20}
\frac{2k+1}{w^2+(2k+1)^2}=\sum_{j=0}^\infty\frac{(-1)^j}{(2k+1)^{2j+1}}w^{2j}.
\end{equation}

Subcase 2(a): $q=2p+1$ for $p\geq 0$. With \eqref{3.16} and \eqref{3.20},
the coefficient of $w^{2p}$ in \eqref{3.12} is
\begin{align*}
&\frac{(-1)^{p-1}}{4^\ell}\sum_{k=0}^\ell\binom{2\ell+1}{\ell-k}\\
&\quad\times
\frac{s_{p-1}((2k+1)a)T_{2k+1}(x)-c_p((2k+1)a)U_{2k}(x)\sqrt{1-x^2}-(-1)^k}
{(2k+1)^{2p+1}}.
\end{align*}
This, with \eqref{3.6}, gives \eqref{2.3}.

Subcase 2(b): $q=2p$ for $p\geq 1$. With \eqref{3.16}, but without 
contribution from \eqref{3.20}, the coefficient of $w^{2p-1}$ in \eqref{3.12} is
\begin{align*}
&\frac{(-1)^{p-1}}{4^\ell}\sum_{k=0}^\ell\binom{2\ell+1}{\ell-k}\\
&\quad\times
\frac{c_{p-1}((2k+1)a)T_{2k+1}(x)-s_{p-1}((2k+1)a)U_{2k}(x)\sqrt{1-x^2}}
{(2k+1)^{2p+1}},
\end{align*}
which gives \eqref{2.5}, once again with \eqref{3.6}. The proof is now complete.
\end{proof}

\section{Some special cases}\label{sec:3}

In order to connect Theorem~\ref{thm:2.1} with the series in \eqref{1.7} and
\eqref{1.9}, we now set $f(x)=(\arcsin{x})^{2p+1}$ with $\delta=1$, resp.\
$f(x)=(\arcsin{x})^{2p}$ with $\delta=0$ in Lemma~\ref{lem:1.1}. Then with
\eqref{1.7} and \eqref{1.9}, respectively, we get the following identities
with the notation \eqref{1.6}.

\begin{lemma}\label{lem:2.2}
For all integers $n\geq 0$ and $p\geq 0$ we have
\begin{equation}\label{2.7}
\sum_{k=0}^{\infty}\frac{\binom{2k}{k}G_p(k)}{4^k}
\cdot\frac{x^{2k+n+2}}{2k+n+2} = \frac{(\arcsin{x})^{2p+1}}{(2p+1)!}x^{n+1}
-\frac{n+1}{(2p+1)!}I_{2p+1}^{(n)}(x),
\end{equation}
and for $p\geq 1$,
\begin{equation}\label{2.8}
\sum_{k=1}^{\infty}\frac{4^kH_p(k)}{\binom{2k}{k}2k}
\cdot\frac{x^{2k+n+1}}{2k+n+1} = \frac{(\arcsin{x})^{2p}}{(2p)!}x^{n+1}
-\frac{n+1}{(2p)!}I_{2p}^{(n)}(x).
\end{equation}
\end{lemma}

The simplest and most direct special case of Lemma~\ref{lem:2.2} is $x=1$,
which leads to the following identities.

\begin{corollary}\label{cor:2.3}
For all integers $\ell\geq 0$ and $p\geq 0$ we have
\begin{align}
\sum_{k=0}^\infty\frac{\binom{2k}{k}G_p(k)}{4^k(2k+2\ell+2)}
&=\frac{(-1)^{p-1}}{4^\ell}\sum_{j=0}^{p-1}
\left(\sum_{k=0}^\ell\frac{\binom{2\ell+1}{\ell-k}}{(2k+1)^{2p-2j}}\right)
\frac{(-1)^j\left(\frac{\pi}{2}\right)^{2j+1}}{(2j+1)!}\label{2.9}\\
&\quad+\frac{(-1)^p}{4^\ell}
\sum_{k=0}^\ell\frac{\binom{2\ell+1}{\ell-k}(-1)^k}{(2k+1)^{2p+1}},\nonumber\\
\sum_{k=0}^\infty\frac{\binom{2k}{k}G_p(k)}{4^k(2k+2\ell+1)}
&=\frac{(-1)^{p-1}}{4^{\ell+p}}\sum_{j=0}^{p-1}
\left(\sum_{k=1}^\ell\frac{\binom{2\ell}{\ell-k}}{k^{2p-2j}}\right)
\frac{(-1)^j\pi^{2j+1}}{(2j+1)!}\label{2.10}\\
&\quad+\frac{\binom{2\ell}{\ell}}{2^{2\ell+2p+1}}\cdot\frac{\pi^{2p+1}}{(2p+1)!},\nonumber
\end{align}
and for $p\geq 1$,
\begin{align}
\sum_{k=1}^{\infty}\frac{4^kH_p(k)}{\binom{2k}{k}2k(2k+2\ell+1)}
&=\frac{(-1)^{p-1}}{4^\ell}\sum_{j=0}^{p-1}
\left(\sum_{k=0}^\ell\frac{\binom{2\ell+1}{\ell-k}}{(2k+1)^{2p-2j}}\right)
\frac{(-1)^j\left(\frac{\pi}{2}\right)^{2j}}{(2j)!},\label{2.11}\\
\sum_{k=1}^{\infty}\frac{4^kH_p(k)}{\binom{2k}{k}2k(2k+2\ell)}
&=\frac{(-1)^{p-1}}{2^{2\ell+2p-1}}\sum_{j=0}^{p-1}
\left(\sum_{k=1}^\ell\frac{\binom{2\ell}{\ell-k}}{k^{2p-2j}}\right)
\frac{(-1)^j\pi^{2j}}{(2j)!}\label{2.12}\\
&\quad+\frac{\binom{2\ell}{\ell}}{4^{\ell+p}}\cdot\frac{\pi^{2p}}{(2p)!}
+\frac{(-1)^p}{2^{2\ell+2p-1}}
\sum_{k=1}^\ell\frac{\binom{2\ell}{\ell-k}(-1)^k}{k^{2p}}.\nonumber
\end{align}
\end{corollary}

\begin{proof}
With $x=1$ we have $a=\arcsin{1}=\pi/2$, and the first identity in \eqref{2.1},
with $\theta=0$, gives $T_n(1)=1$ for all $n\geq 0$. Lemma~\ref{lem:2.2} and
Theorem~\ref{thm:2.1} then give the following four identities:
For all integers $\ell\geq 0$ and $p\geq 0$,
\begin{align*}
\sum_{k=0}^\infty\frac{\binom{2k}{k}G_p(k)}{4^k(2k+2\ell+2)}
&=\frac{(-1)^{p-1}}{4^\ell}\sum_{k=0}^\ell\binom{2\ell+1}{\ell-k}
\frac{s_{p-1}\big((2k+1)\tfrac{\pi}{2}\big)-(-1)^k}{(2k+1)^{2p+1}},\\
\sum_{k=0}^\infty\frac{\binom{2k}{k}G_p(k)}{4^k(2k+2\ell+1)}
&=\frac{\binom{2\ell}{\ell}\left(\frac{\pi}{2}\right)^{2p+1}}{4^\ell(2p+1)!}
-\frac{(-1)^{p}}{2^{2\ell-1}}\sum_{k=1}^\ell\binom{2\ell}{\ell-k}
\frac{s_{p-1}(k\pi)}{(2k)^{2p+1}}, 
\end{align*}
and for $p\geq 1$,
\begin{align*}
\sum_{k=1}^{\infty}\frac{4^kH_p(k)}{\binom{2k}{k}2k(2k+2\ell+1)}
&=\frac{(-1)^{p-1}}{4^\ell}\sum_{k=0}^\ell\binom{2\ell+1}{\ell-k}
\frac{c_{p-1}\big((2k+1)\tfrac{\pi}{2}\big)}{(2k+1)^{2p}}, \\
\sum_{k=1}^{\infty}\frac{4^kH_p(k)}{\binom{2k}{k}2k(2k+2\ell)}
&=\frac{\binom{2\ell}{\ell}\left(\frac{\pi}{2}\right)^{2p}}{4^\ell(2p)!}
-\frac{(-1)^p}{2^{2\ell-1}}\sum_{k=1}^\ell\binom{2\ell}{\ell-k}
\frac{c_{p-1}(k\pi)-(-1)^k}{(2k)^{2p}}. 
\end{align*}
Using \eqref{2.2} and changing the orders of summation, we then get the 
identities \eqref{2.9}--\eqref{2.12}, in this order.
\end{proof}

While in Corollary~\ref{cor:2.3} we fixed a convenient $x$ and kept the 
parameters $\ell$ arbitrary, we are now going to fix the smallest possible
$\ell$ in each case and keep $x$ arbitrary.

\begin{corollary}
For all integers $p\geq 0$ we have, with $a=\arcsin{x}$,
\begin{align}
&\sum_{k=0}^{\infty}\frac{\binom{2k}{k}G_p(k)x^{2k+2}}{4^k(2k+2)}
=(-1)^{p-1}\left(s_{p-1}(a)x+c_p(a)\sqrt{1-x^2}-1\right),\label{2.13}\\
&\sum_{k=0}^{\infty}\frac{\binom{2k}{k}G_p(k)x^{2k+3}}{4^k(2k+3)}
=\frac{a^{2p+1}}{2(2p+1)!} \label{2.14}\\
&\qquad-\frac{(-1)^p}{2^{2p+2}}
\left(s_{p-1}(2a)(2x^2-1)+c_p(2a)2x\sqrt{1-x^2}\right),\nonumber
\end{align}
where $s_{-1}(y)=0$ by convention. Furthermore, for all $p\geq 1$,
\begin{align}
&\sum_{k=1}^{\infty}\frac{4^kH_p(k)x^{2k+1}}{\binom{2k}{k}2k(2k+1)}
=(-1)^{p-1}\left(c_{p-1}(a)x-s_{p-1}(a)\sqrt{1-x^2}\right),\label{2.15}\\
&\sum_{k=1}^{\infty}\frac{4^kH_p(k)x^{2k+2}}{\binom{2k}{k}2k(2k+2)}
=\frac{a^{2p}}{2(2p)!} \label{2.16}\\
&\qquad-\frac{(-1)^p}{2^{2p+1}}
\left(c_{p-1}(2a)(2x^2-1)-s_{p-1}(2a)2x\sqrt{1-x^2}+1\right).\nonumber
\end{align}
\end{corollary}

\begin{proof}
The first two identities follow from setting $\ell=0$ and $\ell=1$, 
respectively, in \eqref{2.3} and \eqref{2.4} and then using \eqref{2.7}.
Similarly, the last two identities follow from \eqref{2.5} and \eqref{2.6},
combined with \eqref{2.8}. In all cases we have used the special cases of
$T_n(x)$ and $U_n(x)$ given after the definitions \eqref{2.1}.
\end{proof}

We see from the right-hand sides of \eqref{2.13}--\eqref{2.15} that we will
get particularly easy identities when the terms in $x$ following $c_{p-1}$ and
$s_{p-1}$ are the same. Indeed, we have the following corollaries.

\begin{corollary}\label{cor:2.5}
For all integers $p\geq 0$ we have
\begin{align}
\sum_{k=0}^{\infty}\frac{\binom{2k}{k}G_p(k)}{8^k(2k+2)}
&=(-1)^{p+1}\sqrt{2}\left(c_{p}(\tfrac{\pi}{4})+s_{p-1}(\tfrac{\pi}{4})\right)
+4(-1)^p\label{2.17}\\
&=(-1)^{p}\left(2-\sqrt{2}\sum_{j=0}^{2p}
\frac{(-1)^{\lfloor\frac{j}{2}\rfloor}}{j!}\left(\frac{\pi}{4}\right)^j\right),\nonumber
\end{align}
and for $p\geq 1$,
\begin{align}
\sum_{k=1}^{\infty}\frac{2^kH_p(k)}{\binom{2k}{k}2k(2k+1)}
&=(-1)^{p-1}\left(c_{p-1}(\tfrac{\pi}{4})-s_{p-1}(\tfrac{\pi}{4})\right)\label{2.18}\\
&=(-1)^{p-1}\sum_{j=0}^{2p-1}
\frac{(-1)^{\lfloor\frac{j+1}{2}\rfloor}}{j!}\left(\frac{\pi}{4}\right)^j.\nonumber
\end{align}
\end{corollary}

\begin{proof}
With $x=\frac{1}{2}\sqrt{2}$ we have $\sqrt{1-x^2}=\frac{1}{2}\sqrt{2}$ as well,
and $a=\arcsin{x}=\pi/4$. The first equalities in \eqref{2.17} and \eqref{2.18}
then follow directly from \eqref{2.13} and \eqref{2.15}, respectively, and the
second equalities follow from the definitions in \eqref{2.2}.
\end{proof}

\begin{example}\label{ex:2.6}
{\rm The two smallest cases in each of \eqref{2.17} and \eqref{2.18} give}
\begin{align*}
\sum_{k=0}^{\infty}\frac{\binom{2k}{k}}{8^k(2k+2)}&=2-\sqrt{2}\\
\sum_{k=0}^{\infty}\frac{\binom{2k}{k}G(k)}{8^k(2k+2)}
&=-2+\sqrt{2}+\frac{\pi}{2\sqrt{2}}-\frac{\pi^2}{16\sqrt{2}},\\
\sum_{k=1}^{\infty}\frac{2^k}{\binom{2k}{k}2k(2k+1)} &= 1-\frac{\pi}{4},\\
\sum_{k=1}^{\infty}\frac{2^kH(k)}{\binom{2k}{k}2k(2k+1)}
&=-1+\frac{\pi}{4}+\frac{\pi^2}{32}-\frac{\pi^3}{384},
\end{align*}
{\rm where $G(k)=G_1(k)$ and $H(k)=H_1(k)$, as in \eqref{4.5} below.}
\end{example}

If we set $x=\frac{1}{2}\sqrt{2}$ in \eqref{2.14} and \eqref{2.16}, we 
have $2x^2-1=0$ and $2x\sqrt{1-x^2}=1$. Then we get the following identities
from \eqref{2.14} and \eqref{2.16}, respectively.

\begin{corollary}\label{cor:2.7}
For all integers $p\geq 0$ we have
\begin{equation}\label{2.19}
\sum_{k=0}^{\infty}\frac{\binom{2k}{k}G_p(k)}{8^k(2k+3)}
=\frac{\sqrt{2}}{(2p+1)!}\left(\frac{\pi}{4}\right)^{2p+1}
-\frac{(-1)^p\sqrt{2}}{2^{2p+1}}c_{p}(\tfrac{\pi}{2}),
\end{equation}
and for $p\geq 1$,
\begin{equation}\label{2.20}
\sum_{k=1}^{\infty}\frac{2^kH_p(k)}{\binom{2k}{k}2k(2k+2)}
=\frac{1}{(2p)!}\left(\frac{\pi}{4}\right)^{2p}
+\frac{(-1)^p}{2^{2p}}\left(s_{p-1}(\tfrac{\pi}{2})-1\right).
\end{equation}
\end{corollary}

\begin{example}\label{ex:2.8}
{\rm The two smallest cases in each of \eqref{2.19} and \eqref{2.20} give}
\begin{align*}
\sum_{k=0}^{\infty}\frac{\binom{2k}{k}}{8^k(2k+3)}
&=\frac{\sqrt{2}}{4}(\pi-2),\\
\sum_{k=0}^{\infty}\frac{\binom{2k}{k}G(k)}{8^k(2k+3)}
&=\frac{\sqrt{2}}{384}\left(\pi^3-6\pi^2+48\right),\\
\sum_{k=1}^{\infty}\frac{2^k}{\binom{2k}{k}2k(2k+2)} 
&=\frac{1}{32}\left(\pi^2-4\pi+8\right),\\
\sum_{k=1}^{\infty}\frac{2^kH(k)}{\binom{2k}{k}2k(2k+2)}
&=\frac{1}{6144}\left(\pi^4-8\pi^3+192\pi+384\right),
\end{align*}
{\rm where $G(k)$ and $H(k)$ are as in \eqref{4.5}.}
\end{example}

Returning to \eqref{2.14} and \eqref{2.16}, we note that with 
$x=\frac{1}{2}\sqrt{2\pm\sqrt{2}}$ we have
\begin{equation}\label{2.21}
2x^2-1=\pm\frac{1}{2}\sqrt{2}\quad\hbox{and}\quad 
2x\sqrt{1-x^2}=\frac{1}{2}\sqrt{2}.
\end{equation}
We also have
\begin{equation}\label{2.22}
\arcsin\left(\frac{1}{2}\sqrt{2-\sqrt{2}}\right)=\frac{\pi}{8},\qquad
\arcsin\left(\frac{1}{2}\sqrt{2+\sqrt{2}}\right)=\frac{3\pi}{8};
\end{equation}
see, e.g., \cite[p.~177]{Sch}. Using the ``$-$" alternative in \eqref{2.21} and
\eqref{2.22}, the identities \eqref{2.14} and \eqref{2.16} yield the following
analogues of \eqref{2.17} and \eqref{2.18} after some straightforward
manipulations.

\begin{corollary}\label{cor:2.9}
For all integers $p\geq 0$ we have
\begin{align}
&\sum_{k=0}^{\infty}\frac{\binom{2k}{k}G_p(k)(2-\sqrt{2})^k}{16^k(2k+3)}\label{2.23}\\
&=\frac{4}{(2-\sqrt{2})^{3/2}}
\left(\frac{1}{(2p+1)!}\left(\frac{\pi}{8}\right)^{2p+1}
-\frac{(-1)^p\sqrt{2}}{4^{p+1}}
\left(c_{p}(\tfrac{\pi}{4})-s_{p-1}(\tfrac{\pi}{4})\right)\right),\nonumber
\end{align}
and for $p\geq 1$,
\begin{align}
&\sum_{k=1}^{\infty}\frac{H_p(k)(2-\sqrt{2})^k}{\binom{2k}{k}2k(2k+2)}\label{2.24}\\
&=\frac{2-\sqrt{2}}{(2p)!}\left(\frac{\pi}{8}\right)^{2p}
+(-1)^p\frac{1+\sqrt{2}}{4^{p}}\left(c_{p-1}(\tfrac{\pi}{4})
+s_{p-1}(\tfrac{\pi}{4})\right)-(-1)^p\frac{2+\sqrt{2}}{4^{p}}.\nonumber
\end{align}
\end{corollary}

As we did in Corollary~\ref{cor:2.5}, the sums and differences of the terms
$c_p$ and $s_p$ could again be written as single sums.

\begin{example}\label{ex:2.10}
{\rm The two smallest cases in each in each of \eqref{2.23} and \eqref{2.24} 
are}
\begin{align*}
\sum_{k=0}^{\infty}\frac{\binom{2k}{k}(2-\sqrt{2})^k}{16^k(2k+3)}
&=\frac{\pi-2\sqrt{2}}{2(2-\sqrt{2})^{3/2}},\\
\sum_{k=0}^{\infty}\frac{\binom{2k}{k}G(k)(2-\sqrt{2})^k}{16^k(2k+3)}
&=\frac{1}{(2-\sqrt{2})^{3/2}}\left(\frac{\pi^3}{768}
-\frac{\sqrt{2}}{128}\left(\pi^2+8\pi-32\right)\right),\\
\sum_{k=1}^{\infty}\frac{(2-\sqrt{2})^k}{\binom{2k}{k}2k(2k+2)}
&=\frac{2+\sqrt{2}}{128}\pi^2-\frac{1+\sqrt{2}}{16}\pi+\frac{1}{4},\\
\sum_{k=1}^{\infty}\frac{H(k)(2-\sqrt{2})^k}{\binom{2k}{k}2k(2k+2)}
&=\frac{2+\sqrt{2}}{98304}\pi^4-\frac{1+\sqrt{2}}{6144}
\left(\pi^3+12\pi^2-96\pi\right)-\frac{1}{16}.
\end{align*}
\end{example}

Using the ``$+$" choice in \eqref{2.21} and the second identity in 
\eqref{2.22}, we could obtain analogues to Corollary~\ref{cor:2.9} and
Example~\ref{ex:2.10}, involving powers of $2+\sqrt{2}$. Further special cases
are also possible by using other sine (or arcsine) evaluation such as 
$\sin(\pi/10)=(\sqrt{5}-1)/4$ or $\sin(\pi/12)=(\sqrt{6}-\sqrt{2})/4$; see
\cite[p.~177f]{Sch}.

\section{Some classes of limit expressions}\label{sec:4}

Since the four identities in Corollary~\ref{cor:2.3} hold for all integers
$\ell\geq 0$, it is a reasonable question to ask what happens as 
$\ell\to\infty$ for a fixed $p$. For $p=1,\ldots,4$, this was answered in
\cite{DV}: We have the following limits as $n\rightarrow\infty$.
\begin{align}
\frac{2^{n+1}}{\binom{n}{n/2}}\sum_{k=0}^\infty
&\frac{\binom{2k}{k}}{4^k(2k+n+1)}\rightarrow\pi,\label{4.1}\\
\frac{2^{n+1}}{\binom{n}{n/2}}\sum_{k=1}^\infty
&\frac{4^k}{\binom{2k}{k}k(2k+n)}\rightarrow\pi^2,\label{4.2}\\
2\cdot\frac{2^{n+3}}{\binom{n}{n/2}}\sum_{k=0}^\infty
&\frac{\binom{2k}{k}G(k)}{4^k(2k+n+1)}\rightarrow\pi^3,\label{4.3}\\
6\cdot\frac{2^{n+3}}{\binom{n}{n/2}}\sum_{k=1}^\infty
&\frac{4^kH(k)}{\binom{2k}{k}k(2k+n)}\rightarrow\pi^4,\label{4.4}
\end{align}
where 
\begin{equation}\label{4.5}
G(k):=G_1(k)=\sum_{j=0}^{k-1}\frac{1}{(2j+1)^2}\quad\hbox{and}\quad
H(k):=H_2(k)=\sum_{j=1}^{k-1}\frac{1}{(2j)^2},
\end{equation}
and where the central binomial coefficient $\binom{n}{n/2}$ also makes sense
for odd $n=2\ell+1$ if interpreted as
\begin{equation}\label{4.6}
\binom{n}{n/2}=\binom{2\ell+1}{\ell+\tfrac{1}{2}}
=\frac{\Gamma(2\ell+2)}{\Gamma(\ell+\tfrac{3}{2})^2}
=\frac{4^{2\ell+1}}{\binom{2\ell}{\ell}(2\ell+1)\pi}.
\end{equation}
This can be obtained from the well-known identity
\[
\Gamma(m+\tfrac{1}{2})=\frac{(2m)!}{4^m m!}\sqrt{\pi},
\]
valid for integers $m\geq0$.

We can now state and prove the main result of this section.

\begin{theorem}\label{thm:4.1}
For any integer $p\geq 0$ we have the following limits as $n\to\infty$:
\begin{equation}\label{4.7}
\frac{2^{2p+n+1}(2p)!}{\binom{n}{n/2}}\sum_{k=0}^\infty
\frac{\binom{2k}{k}G_p(k)}{4^k(2k+n+1)}\rightarrow\pi^{2p+1},
\end{equation}
and for $p\geq 1$,
\begin{equation}\label{4.8}
\frac{2^{2p+n-1}(2p-1)!}{\binom{n}{n/2}}\sum_{k=1}^\infty
\frac{4^kH_p(k)}{\binom{2k}{k}k(2k+n)}\rightarrow\pi^{2p},
\end{equation}
where $G_p(k)$ and $H_p(k)$ are as defined in \eqref{1.8} and \eqref{1.10}, and
for odd $n$ the binomial coefficient $\binom{n}{n/2}$ is interpreted as in
\eqref{4.6}.
\end{theorem}

We see that \eqref{4.7} with $p=0$ and $p=1$ immediately gives \eqref{4.1} and
\eqref{4.3}, respectively, while \eqref{4.8} with $p=1$ and $p=2$ recovers 
\eqref{4.2} and \eqref{4.4}.

For the proof of Theorem~\ref{4.1}, we require a few lemmas.

\begin{lemma}\label{lem:4.2}
Let $s\in{\mathbb R}$, $s>1$. Then we have the following limits as 
$\ell\to\infty$:
\begin{align}
\frac{1}{\binom{2\ell}{\ell}}\sum_{k=1}^\ell\binom{2\ell}{\ell-k}\frac{1}{k^s}
&\to\zeta(s),\label{4.9}\\
\frac{1}{\binom{2\ell}{\ell}}
\sum_{k=1}^\ell\binom{2\ell}{\ell-k}\frac{(-1)^{k-1}}{k^s}
&\to\left(1-2^{1-s}\right)\zeta(s),\label{4.10}\\
\frac{1}{\binom{2\ell+1}{\ell+1}}
\sum_{k=0}^\ell\binom{2\ell+1}{\ell-k}\frac{1}{(2k+1)^s}
&\to\left(1-2^{-s}\right)\zeta(s),\label{4.11}\\
\frac{1}{\binom{2\ell+1}{\ell+1}}
\sum_{k=0}^\ell\binom{2\ell+1}{\ell-k}\frac{(-1^k)}{(2k+1)^s}
&\to \beta(s),\label{4.12}
\end{align}
where $\zeta(s)$ is the Riemann zeta function and $\beta(s)$ is the Dirichlet 
beta function, respectively defined by
\[
\zeta(s)=\sum_{k=1}^\infty\frac{1}{k^s}\quad({\rm Re}(s)>1)\quad\hbox{and}\quad
\beta(s)=\sum_{k=0}^\infty\frac{(-1)^k}{(2k+1)^s}\quad({\rm Re}(s)>0).
\]
\end{lemma}

\begin{proof}
It is easy to see that for a fixed $k$ we have
\[
\frac{\binom{2\ell}{\ell-k}}{\binom{2\ell}{\ell}}
=\prod_{j=0}^{k-1}\frac{\ell-j}{\ell+1-j}\to 1,\quad\hbox{and}\quad
\frac{\binom{2\ell+1}{\ell-k}}{\binom{2\ell+1}{\ell+1}}
=\prod_{j=0}^{k-1}\frac{\ell-j}{\ell+2-j}\to 1,
\]
as $\ell\to\infty$. By Tannery's theorem (see, e.g., \cite[p.~136]{Br} or the
Appendix in \cite{Ho}), the limit of the sum in \eqref{4.9} is then $\zeta(s)$,
as desired. Similarly, the limits of the sums in \eqref{4.10} and \eqref{4.11} 
are respectively
\begin{align*}
\sum_{k=0}^\infty\frac{(-1)^{k-1}}{k^s}
&=\sum_{j=0}^\infty\frac{1}{j^s}-2\sum_{j=0}^\infty\frac{1}{(2j)^s}
=\left(1-2^{1-s}\right)\zeta(s),\\
\sum_{k=0}^\infty\frac{1}{(2k+1)^s}
&=\sum_{j=0}^\infty\frac{1}{j^s}-\sum_{j=0}^\infty\frac{1}{(2j)^s}
=\left(1-2^{-s}\right)\zeta(s),
\end{align*}
again as desired, and \eqref{4.12} follows similarly. However, while $\beta(s)$
is defined for all $s>0$, Tannery's theorem applies only for $s>1$ in this case.
\end{proof}

Important tools in dealing with zeta and related functions are the Bernoulli
and Euler numbers, which are usually defined by the exponential generating 
functions
\begin{equation}\label{4.13}
\frac{t}{e^t-1} = \sum_{n=0}^\infty B_n\frac{t^n}{n!}\qquad\hbox{and}\qquad
\frac{2}{e^t+e^{-t}} = \sum_{n=0}^\infty E_n\frac{t^n}{n!}.
\end{equation}
They satisfy, in particular,
\begin{equation}\label{4.14}
B_0=E_0=1,\quad B_1=-\frac{1}{2},\quad\hbox{and}\quad
B_{2j+1}=E_{2j-1}=0\quad\hbox{for all}\quad j\geq 1.
\end{equation}
We will also use the Bernoulli and Euler polynomials, which can be defined by
\begin{equation}\label{4.15}
B_n(x) = \sum_{j=0}^n\binom{n}{j}B_j x^{n-j},\qquad
E_{n-1}(x)=\frac{2}{n}\sum_{j=0}^n\binom{n}{j}\left(1-2^j\right)B_j x^{n-j},
\end{equation}
respectively. They have the special values
\begin{equation}\label{4.16}
B_n(\tfrac{1}{2}) = \left(2^{1-n}-1\right)B_n,\qquad
E_n(\tfrac{1}{2}) = 2^{-n}E_n.
\end{equation}
All these properties can be found, for instance, in \cite[Sec.~9.6]{GR} or
\cite[Ch.~24]{DLMF}. The main connection with the proof of Theorem~\ref{thm:4.1}
is given through Euler's well-known identity
\begin{equation}\label{4.17}
\zeta(2n)=(-1)^{n-1}\frac{(2\pi)^{2n}}{2(2n)!}B_{2n},\qquad n=1, 2, \ldots
\end{equation}
(see, e.g., \cite[eq.~9.616]{GR}) and its analogue
\begin{equation}\label{4.18}
\beta(2n+1)=(-1)^n\frac{\pi^{2n+1}}{2^{2n+2}(2n)!}E_{2n},\qquad n=0, 1, \ldots
\end{equation}
(see, e.g., \cite[eq.~0.233.6]{GR}). These last two identities are, in fact, 
special cases of a general identity for Dirichlet $L$-series at either even
or odd positive integers.

In the proof of Theorem~\ref{thm:4.1} we also require the following summation
formulas for Bernoulli numbers.

\begin{lemma}\label{lem:4.3}
For all integers $p\geq 1$ we have
\begin{align}
&\sum_{j=0}^{p-1}\binom{2p+1}{2j+1}2^{2p-2j}B_{2p-2j}=2p,\label{4.19}\\
&\sum_{j=0}^{p-1}\binom{2p+1}{2j+1}2^{2p-2j}\left(2^{2p-2j}-1\right)B_{2p-2j}
=(2p+1)(1-E_{2p}),\label{4.20}\\
&\sum_{j=0}^{p-1}\binom{2p}{2j}2^{2p-2j}B_{2p-2j}
=2p-1+\left(2-2^{2p}\right)B_{2p},\label{4.21}\\
&\sum_{j=0}^{p-1}\binom{2p}{2j}2^{2p-2j}\left(2^{2p-2j}-1\right)B_{2p-2j}
=2p.\label{4.22}
\end{align}
\end{lemma}

\begin{proof}
All four identities can be derived in basically the same way: We first factor
out the term $2^{2p}$ and reverse the order of summation. Then, using the
properties in \eqref{4.14}, we write the sums either as Bernoulli or as Euler
polynomials, as given in \eqref{4.15}.
Denoting the sums in \eqref{4.19}--\eqref{4.22} by $S_1,\ldots, S_4$, 
respectively, we then get
\begin{align*}
S_1&=\sum_{j=1}^p\binom{2p+1}{2j}\left(\frac{1}{2}\right)^{2p-2j}B_{2j} \\
&=2^{2p}\left(\sum_{j=0}^{2p+1}\binom{2p+1}{j}\left(\frac{1}{2}\right)^{2p-j}
B_j -\left(\frac{1}{2}\right)^{2p}
-(2p+1)\left(\frac{1}{2}\right)^{2p-1}\left(\frac{-1}{2}\right)\right) \\
&=2^{2p}\left(2B_{2p+1}(\tfrac{1}{2})+\frac{2p}{2^{2p}}\right)=2p,
\end{align*}
where we have used the first identity in \eqref{4.16} and \eqref{4.14}. Next,
\begin{align*}
S_2&=\sum_{j=1}^p\binom{2p+1}{2j}\left(2^{2j}-1\right)
\left(\frac{1}{2}\right)^{2p-2j}B_{2j} \\
&=2^{2p}\left(\sum_{j=0}^{2p+1}\binom{2p+1}{j}\left(2^{2j}-1\right)
\left(\frac{1}{2}\right)^{2p-j}B_j 
-(2p+1)\left(\frac{1}{2}\right)^{2p-1}\left(\frac{-1}{2}\right)\right) \\
&=2^{2p}\left(-2\frac{2p+1}{2}B_{2p}(\tfrac{1}{2})+\frac{2p+1}{2^{2p}}\right)
=(2p+1)(1-E_{2p}),
\end{align*}
where we have used the second identity in \eqref{4.16}. Next, in analogy to
$S_1$ we have
\begin{align*}
S_3&=2^{2p}\left(\sum_{j=0}^{2p}\binom{2p}{j}\left(\frac{1}{2}\right)^{2p-j}B_j
-\left(\frac{1}{2}\right)^{2p}
-2p\left(\frac{1}{2}\right)^{2p-1}\left(\frac{-1}{2}\right)\right) \\
&=2^{2p}\left(2B_{2p}(\tfrac{1}{2})+\frac{2p-1}{2^{2p}}\right)
=\left(2-2^{2p}\right)B_{2p}+2p-1,
\end{align*}
where we have used the first identity in \eqref{4.16}. Finally, in analogy to
$S_2$ we have
\begin{align*}
S_4&=2^{2p}\left(\sum_{j=0}^{2p}\binom{2p}{j}\left(2^{2j}-1\right)
\left(\frac{1}{2}\right)^{2p-j}B_j
-2p\left(\frac{1}{2}\right)^{2p-1}\left(\frac{-1}{2}\right)\right) \\
&=2^{2p}\left(-2\frac{2p}{2}E_{2p-1}(\tfrac{1}{2})+\frac{2p}{2^{2p}}\right)=2p,
\end{align*}
where we have again used the second identity in \eqref{4.16}, followed by
\eqref{4.14}. This completes the proof.
\end{proof}

\begin{proof}[Proof of Theorem~\ref{thm:4.1}]
To prove \eqref{4.7}, we first multiply both sides of \eqref{2.10} by
$2^{2p+2\ell+1}(2p)!/\binom{2\ell}{\ell}$ and denote the resulting expression
by $X_\ell$. Then we have
\[
X_\ell=(-1)^{p-1}2(2p)!\sum_{j=0}^{p-1}
\left(\sum_{k=1}^\ell\frac{\binom{2\ell}{\ell-k}}{\binom{2\ell}{\ell}}
\frac{1}{k^{2p-2j}}\right)\frac{(-1)^j\pi^{2j+1}}{(2j+1)!}
+\frac{\pi^{2p+1}}{2p+1}.
\]
Denoting $X:=\lim_{\ell\to\infty}X_\ell$, we get with \eqref{4.9} and
\eqref{4.17},
\begin{align*}
X&=2(2p)!\sum_{j=0}^{p-1}\frac{(2\pi)^{2p-2j}}{2(2p-2j)!}B_{2p-2j}
\frac{\pi^{2j+1}}{(2j+1)!}+\frac{\pi^{2p+1}}{2p+1} \\
&=\frac{\pi^{2p+1}}{2p+1}\left(1+
\sum_{j=0}^{p-1}\binom{2p+1}{2j+1}2^{2p-2j}B_{2p-2j}\right)=\pi^{2p+1},
\end{align*}
where the last identity comes from \eqref{4.19}. This proves the limit 
\eqref{4.7} for even $n$.

Next, we multiply both sides of \eqref{2.9} by
$2^{2p+2\ell+2}(2p)!/\binom{2\ell+1}{\ell+1}$ and denote the resulting 
expression by $Y_\ell$. Then we have
\begin{align*}
Y_\ell&=(-1)^{p-1}2^{2p+2}(2p)!\sum_{j=0}^{p-1}
\left(\sum_{k=0}^\ell\frac{\binom{2\ell+1}{\ell-k}}{\binom{2\ell+1}{\ell+1}}
\frac{1}{(2k+1)^{2p-2j}}\right)
\frac{(-1)^j\left(\tfrac{\pi}{2}\right)^{2j+1}}{(2j+1)!}\\
&\quad+(-1)^p2^{2p+2}(2p)!\sum_{k=0}^\ell
\frac{\binom{2\ell+1}{\ell-k}}{\binom{2\ell+1}{\ell+1}}
\frac{(-1)^k}{(2k+1)^{2p+1}}.
\end{align*}
Denoting $Y:=\lim_{\ell\to\infty}Y_\ell$, we get with \eqref{4.11}, \eqref{4.12}
and with \eqref{4.17}, \eqref{4.18},
\begin{align*}
Y&=2^{2p+2}(2p)!\sum_{j=0}^{p-1}\left(1-2^{2j-2p}\right)
\frac{(2\pi)^{2p-2j}}{2(2p-2j)!}B_{2p-2j}\frac{\pi^{2j+1}}{2^{2j+1}(2j+1)!}\\
&\qquad+2^{2p+2}(2p)!\frac{\pi^{2p+1}}{2^{2p+2}(2p)!}E_{2p}\\
&=\pi^{2p+1}\left(E_{2p}+\frac{1}{2p+1}\sum_{j=0}^{p-1}\left(2^{2p-2j}-1\right)
\binom{2p+1}{2j+1}2^{2p-2j}B_{2p-2j}\right)=\pi^{2p+1},
\end{align*}
where the last identity comes from \eqref{4.19}. Furthermore, we have by
\eqref{4.6},
\begin{equation}\label{4.23}
\frac{\binom{2\ell+1}{\ell+1}}{\binom{2\ell+1}{\ell+1/2}}
=\frac{(2\ell+1)^2}{4(\ell+1)}\frac{\binom{2\ell}{\ell}^2}{4^{2\ell}}\pi,
\end{equation}
and by the well-known asymptotics $\binom{2\ell}{\ell}\sim 4^\ell/\sqrt{\pi k}$,
the expression \eqref{4.23} approaches 1 as $\ell\to\infty$. This, together
with the fact that $X_\ell\to\pi^{2p+1}$, proves the limit \eqref{4.7} for
odd $n$.

To prove \eqref{4.8}, we proceed analogously. We multiply both sides of 
\eqref{2.12} by $2^{2p+2\ell}(2p-1)!/\binom{2\ell}{\ell}$ and denote the 
resulting expression by $Z_\ell$. Then we have
\begin{align*}
Z_\ell&=(-1)^{p-1}2(2p-1)!\sum_{j=0}^{p-1}
\left(\sum_{k=1}^\ell\frac{\binom{2\ell}{\ell-k}}{\binom{2\ell}{\ell}}
\frac{1}{k^{2p-2j}}\right)\frac{(-1)^j\pi^{2j}}{(2j)!}\\
&\quad+\frac{\pi^{2p}}{2p}+(-1)^{p-1}2(2p-1)!
\sum_{k=1}^\ell\frac{\binom{2\ell}{\ell-k}}{\binom{2\ell}{\ell}}
\frac{(-1)^{k-1}}{k^{2p}}.
\end{align*}
Denoting $Z:=\lim_{\ell\to\infty}Z_\ell$, we get with \eqref{4.9}, \eqref{4.10}
and \eqref{4.17},
\begin{align*}
Z&=2(2p-1)!\sum_{j=0}^{p-1}\frac{(2\pi)^{2p-2j}}{2(2p-2j)!}B_{2p-2j}
\frac{\pi^{2j}}{(2j)!}+\frac{\pi^{2p}}{2p} \\
&\qquad+2(2p-1)!\left(1-2^{1-2p}\right)\frac{(2\pi)^{2p}}{2(2p)!}B_{2p}\\
&=\frac{\pi^{2p}}{2p}\left(1+\left(2^{2p}-2\right)B_{2p}
+\sum_{j=0}^{p-1}\binom{2p}{2j}2^{2p-2j}B_{2p-2j}\right)=\pi^{2p},
\end{align*}
where the last identity comes from \eqref{4.21}. This proves the limit
\eqref{4.8} for even $n$.

To deal with odd $n=2\ell+1$, we proceed again as before: Multiply both sides 
of \eqref{2.11} by $2^{2p+2\ell}(2p-1)!/\binom{2\ell+1}{\ell+1}$, take the
limit as $\ell\to\infty$ and denote it by $W$. Then we get with \eqref{4.11} 
and \eqref{4.17},
\begin{align*}
W&=2^{2p+1}(2p-1)!\sum_{j=0}^{p-1}\left(1-2^{2j-2p}\right)
\frac{(2\pi)^{2p-2j}}{2(2p-2j)!}B_{2p-2j}\frac{\pi^{2j}}{2^{2j}(2j)!}\\
&=\frac{\pi^{2p}}{2p}\sum_{j=0}^{p-1}\binom{2p}{2j}\left(2^{2p-2j}-1\right)
2^{2p-2j}B_{2p-2j}=\pi^{2p},
\end{align*}
where the last identity comes from \eqref{4.22}. Using \eqref{4.23}, as we did
in the proof of \eqref{4.7}, then shows that the limit \eqref{4.8} also holds
for odd $n$. The proof is now complete. 
\end{proof}

\section{Further remarks and consequences}\label{sec:5}

{\bf 1.} In an unpublished preprint, C.~Lupu \cite{Lu} used direct 
manipulations of the arcsine identities \eqref{1.7} and \eqref{1.9} to obtain
series evaluations that are quite similar to some of our corollaries in 
Section~\ref{sec:3}. For example, he proved (rewritten in our notation and
normalization) that for all integers $p\geq 0$, 
\[
\sum_{k=0}^{\infty}\frac{\binom{2k}{k}G_p(k)}{8^k}
=\frac{1}{(2p)!\sqrt{2}}\left(\frac{\pi}{4}\right)^{2p}
\]
and for $p\geq 1$,
\[
\sum_{k=1}^{\infty}\frac{2^kH_p(k)}{\binom{2k}{k}k}
=\frac{2\sqrt{2}}{(2p-1)!}\left(\frac{\pi}{4}\right)^{2p-1},
\]
which can be seen as supplementing our Corollary~\ref{cor:2.7}. These and the
six other such evaluations obtained in \cite{Lu} are of a particularly pleasing
form in that the right-hand sides are simple multiples of powers of $\pi$.

{\bf 2.} Using partial fraction expansions, we can obtain further series
evaluation from some of the identities in Sections~\ref{sec:2} and~\ref{sec:3}.
For instance, with
\begin{equation}\label{5.1}
\frac{1}{(2k+2+m)(2k+2+n)}=\frac{1}{n-m}
\left(\frac{1}{2k+2+m}-\frac{1}{2k+2+n}\right)
\end{equation}
($n\neq m$) combined with the identity \eqref{3.1}, we obtain the following.

\begin{corollary}\label{cor:5.1}
For all nonnegative integers $p$ and $n\neq m$ we have
\begin{equation}\label{5.2}
\sum_{k=0}^{\infty}\frac{\binom{2k}{k}G_p(k)x^{2k+2}}{4^k(2k+2+n)(2k+2+m)} 
=\frac{(n+1)x^{-n}I_{2p+1}^{(n)}(x)-(m+1)x^{-m}I_{2p+1}^{(m)}(x)}{(n-m)(2p+1)!}.
\end{equation}
\end{corollary}

\begin{example}\label{ex:5.2}
{\rm With $x=1, n=1$ and $m=2$ we get from \eqref{5.2},
\begin{equation}\label{5.3}
\sum_{k=0}^{\infty}\frac{\binom{2k}{k}G_p(k)}{4^k(2k+3)(2k+4)}
=\frac{3I_{2p+1}^{(2)}(1)-2I_{2p+1}^{(1)}(1)}{(2p+1)!}.
\end{equation}
Using evaluations as in Corollary~\ref{cor:2.3}, the identity \eqref{5.3} 
yields with $p=0$ and $p=1$, respectively,}
\begin{equation}\label{5.4}
\sum_{k=0}^{\infty}\frac{\binom{2k}{k}}{4^k(2k+3)(2k+4)}
= \frac{\pi}{4}-\frac{2}{3},
\end{equation}
\begin{equation}\label{5.5}
\sum_{k=0}^{\infty}\frac{\binom{2k}{k}G_1(k)}{4^k(2k+3)(2k+4)}
= \frac{\pi^3}{96}-\frac{47\pi}{144}+\frac{20}{27}.
\end{equation}
\end{example}

\begin{example}\label{ex:5.3}
{\rm We choose $n=1$ and $m=2$ again and use \eqref{5.2} with \eqref{2.3} and
\eqref{2.4}. Then we get more generally for $p=0$,
\begin{align*}
\sum_{k=0}^{\infty}\frac{\binom{2k}{k}x^{2k+3}}{4^k(2k+2+n)(2k+2+m)}
&=-\frac{2}{3}+\left(\frac{x^2}{3}-\frac{x}{2}+\frac{2}{3}\right)\sqrt{1-x^2}\\
&\quad +\left(x^3-x^2+\frac{1}{2}\right)\arcsin{x},
\end{align*}
and for $p=1$,
\begin{align*}
&\sum_{k=0}^{\infty}\frac{\binom{2k}{k}G_1(k)x^{2k+3}}{4^k(2k+2+n)(2k+2+m)}
=\frac{20}{27}-\left(\frac{x^2}{27}-\frac{x}{8}+\frac{20}{27}\right)\sqrt{1-x^2}\\
&\quad-\left(\frac{x^3}{9}-\frac{x^2}{4}+\frac{2x}{3}+\frac{1}{8}\right)\arcsin{x}
+\left(\frac{x^2}{6}-\frac{x}{4}+\frac{1}{3}\right)\sqrt{1-x^2}\cdot\arcsin^2{x}\\
&\quad+\left(\frac{x^3}{6}-\frac{x^2}{6}+\frac{1}{12}\right)\arcsin^3{x}.
\end{align*}
Setting $x=1$, we recover the identities \eqref{5.4} and \eqref{5.5}.}
\end{example}

It is clear that we can obtain more identities by using \eqref{3.2} to get an
analogue of Corollary~\ref{cor:5.1}, or by using partial fraction expansions
of more than three factors in place of \eqref{5.1}.

\medskip
{\bf 3.} The multiple sums $H_p(k)$ and $G_p(k)$, defined in \eqref{1.10} and
\eqref{1.8}, are special cases of so-called multiple harmonic and multiple
$t$-harmonic sums which have recently been studied quite intensively.
In fact, we have
\[
H_{p+1}=4^{-p}H_{k-1}(\{2\}^p),
\]
where $\{2\}^p$ denotes the $p$-tuple $(2,2,\ldots,2)$; see, e.g., \cite{HHT}.
Second, we have 
\[
G_p(k) = t_k(\{2\}^p);
\]
see, e.g., \cite{LY}. The corresponding multiple sums with non-strict 
inequalities between their summation indices are called multiple harmonic star
sums and multiple $t$-harmonic star sums, respectively. They were the main
object of study in \cite{HHT} and \cite{LY}. We caution that our use use of 
the letter $p$ above should not be confused with the prime $p$ used in 
\cite{HHT} and in other related publications.

\medskip
{\bf 4.} As mentioned in the Introduction, we used a different method in 
\cite{DV} to evaluate $I_q^{(n)}(x)$ for $q=1,\ldots, 4$. The results are, 
in fact, of different forms, and by equating them we get some nontrivial 
combinatorial identities such as the following.

\begin{corollary}\label{cor:5.4}
For any integer $\ell\geq 1$ we have
\begin{align}
4\sum_{k=0}^{\lfloor\frac{\ell-1}{2}\rfloor}
\binom{2\ell}{\ell-2k-1}\frac{1}{(2k+1)^2}
&=\binom{2\ell}{\ell}\sum_{j=1}^\ell\frac{4^j}{\binom{2j}{j}j^2},\label{5.6} \\
\sum_{k=0}^\ell\binom{2\ell+1}{\ell-k}\frac{1}{(2k+1)^2}
&=\frac{4^{2\ell}}{\binom{2\ell}{\ell}(2\ell+1)}
\sum_{j=0}^\ell\frac{\binom{2j}{j}}{4^j(2j+1)}.\label{5.7}
\end{align}
\end{corollary}

\begin{proof}
The rational part of \eqref{2.12} gives, after some simplification,
\[
\frac{1}{2^{2\ell}}\sum_{k=0}^{\lfloor\frac{\ell-1}{2}\rfloor}
\binom{2\ell}{\ell-2k-1}\frac{1}{(2k+1)^2}.
\]
On the other hand, the rational part of the same series in Corollary~5 in
\cite{DV} gives
\[
\frac{1}{2^{2\ell+2}}\binom{2\ell}{\ell}
\sum_{j=1}^\ell\frac{4^j}{\binom{2j}{j}j^2}.
\]
By equating the two, we immediately get \eqref{5.6}.

The identity \eqref{5.7} is obtained analogously by equating the rational part
of \eqref{2.11} in the case $p=1$ with that of the second identity in 
Corollary~5 of \cite{DV}. (A note of caution: The notation $I_q^{(n)}(x)$ in
\cite{DV} is slightly different from that in the present paper.)
\end{proof} 

We can obtain similar identities by taking $p=2$ in \eqref{2.11} and 
\eqref{2.12}. It turns out that, along with \eqref{5.6} and \eqref{5.7}, those
are the first cases of general identities for some new variants of multiple
harmonic sums and related multiple zeta functions. However, this goes beyond
the scope of the present paper and will be the subject of a separate paper
\cite{DV3}.

We finish with the observation that after dividing both sides of \eqref{5.6}
by $\binom{2\ell}{\ell}$, the limit as $\ell\to\infty$ exists, and by 
Tannery's theorem (see also Lemma~\ref{lem:4.2}) we get
\begin{equation}\label{5.8}
4\sum_{k=0}^{\infty}\frac{1}{(2k+1)^2}
=\sum_{j=1}^\infty\frac{4^j}{\binom{2j}{j}j^2}.
\end{equation}
The identity \eqref{5.6} can therefore be seen as a finite analogue of
the equality between the infinite series in \eqref{5.8}.

This gets us back to the beginning if we note that the series on the left of
\eqref{5.8} evaluates as $(1-2^{-2})\zeta(2)$ (see the proof of 
Lemma~\ref{lem:4.2}) and is therefore equal to $\pi^2/2$. This turns 
\eqref{5.8} into a well-known identity which is in fact the special case 
$x=1$ in the second identity of \eqref{1.1}.

\end{document}